\documentclass[10pt, reqno]{amsart}
\usepackage{amsmath, amsthm, amscd, amsfonts, amssymb, graphicx, color}
\usepackage[bookmarksnumbered, colorlinks, plainpages]{hyperref}
\hypersetup{colorlinks=true,linkcolor=red, anchorcolor=green, citecolor=cyan, urlcolor=red, filecolor=magenta, pdftoolbar=true}
\usepackage{mathrsfs}
\newtheorem{theorem}{Theorem}[section]
\newtheorem{lemma}[theorem]{Lemma}

\newtheorem{corollary}[theorem]{Corollary}
\theoremstyle{definition}
\newtheorem{definition}[theorem]{Definition}

\theoremstyle{remark}
\newtheorem{remark}[theorem]{Remark}
\numberwithin{equation}{section}

\begin{document}
\setcounter{page}{1}

\title[Elliptic curves and continued fractions]{Elliptic curves and continued fractions}

\author[Nikolaev]
{Igor ~V. ~Nikolaev$^1$}

\address{$^{1}$ Department of Mathematics and Computer Science, St.~John's University, 8000 Utopia Parkway,  
New York,  NY 11439, United States.}
\email{\textcolor[rgb]{0.00,0.00,0.84}{igor.v.nikolaev@gmail.com}}

%\dedicatory{All data is available as part of this manuscript}

%\dedicatory{In memory of Ola Bratteli}

\subjclass[2010]{Primary 11G05; Secondary 46L85.}

\keywords{elliptic curves, non-commutative tori.}

%\date{Received:  August 14, 2015; Revised: yyyyyy; Accepted: zzzzzz.}

\begin{abstract}
It is proved that  the  rank of an elliptic curve  is one less 
the  arithmetic complexity of the corresponding non-commutative
torus.   As an illustration, we consider a family  of elliptic curves with complex multiplication.  
\end{abstract} 

\maketitle

%**************************************************************************
\section{Introduction}
%***************************************************************************
Let $\mathcal{E}(K)$ be an  elliptic curve over  the number field $K$.
The Mordell Theorem  says that $\mathcal{E}(K)$ is  a finitely generated abelian group, 
see e.g.  [Silverman \& Tate  1992]  \cite[Chapter 1]{ST}. 
 Little is known about the rank  $rk~\mathcal{E}(K)$ of $\mathcal{E}(K)$ in general,  except for the 
 Birch and Swinnerton-Dyer Conjecture  comparing the rank 
 with the order of zero at the point $s=1$  of the Hasse-Weil $L$-function $L(\mathcal{E},s)$  
 of $\mathcal{E}(K)$.

\smallskip
The non-commutative torus  $\mathcal{A}_{\theta}$  is a 
$C^*$-algebra generated by the unitary operators $u$ and $v$ satisfying 
the commutation relation $vu=e^{2\pi i\theta} uv$ for  a real constant $\theta$
 [Rieffel 1990]  \cite{Rie1}.   The algebra $\mathcal{A}_{\theta}$ is said to have  real multiplication
   if $\theta$ is a quadratic irrationality.   We shall denote
such an algebra  by $\mathcal{A}_{RM}$.  
There exists a covariant functor $F$ from elliptic curves to non-commutative tori, such that 
$F(\mathcal{E}(K))=\mathcal{A}_{RM}$  \cite[Section 1.3]{N}.   Functor  $F$ maps elliptic
curves $\mathcal{E}(K)$ which are isomorphic over $K$ (over the algebraic closure of $K$, resp.)
to isomorphic (Morita equivalent, resp.) non-commutative tori  $\mathcal{A}_{RM}$, 
see also remark \ref{rmk3.2}.

\smallskip
The aim of our note is a formula for $rk~\mathcal{E}(K)$  in terms of
the invariants of 
 the  algebra $\mathcal{A}_{RM}$. Recall that 
each quadratic irrationality can be written in the form 
$\theta_d={a+b\sqrt{d}\over c}$, where $a,b,c\in \mathbf{Z}, ~bc\ne 0$ and $d>0$ is a square-free
integer.   The  regular continued fraction of $\theta_d$  is  known to be eventually periodic,
i.e. $\theta_d=[g_1,\dots, g_m, ~\overline{k_1,\dots, k_{n-m}}]$,  where 
$(k_1,\dots, k_{n-m})$ is  the minimal period and $n\ge m+1\ge 1$.   
Consider a family $\theta_x$ of the irrational numbers of the form
%*************************************************************************
\begin{equation}\label{eq1.1}
\theta_x:=\left\{ {a+b\sqrt{x}\over c} ~| ~ a,b,c=Const\right\},
\end{equation}
%********************************************************************************
where $x$ runs through a  set $U_d$ of the square-free positive integers containing $d$. 
The family of  tori $\{\mathcal{A}_{\theta_x} ~|~x\in U_d\}$ is  a non-commutative analog 
of a rational  elliptic surface  [Silverman 1994]  \cite[Chapter III]{S}, see also remark \ref{rmk3.3.1}.

By the Euler equations we understand a system of polynomial relations in the ring 
 $\mathbf{Z}[g_1,\dots, g_m; ~k_1,\dots,k_{n-m}]$ which solve   the equation
 %********************************************************************************
 \begin{equation}\label{eq1.1.1}
 \theta_x=[g_1(x),\dots, g_m(x),  ~\overline{k_1(x),\dots, k_{n-m}(x)}]
 \end{equation}
 %***********************************************************************************
 for some  polynomials $g_i(x)$ and  $k_j(x)$  in $\mathbf{Z}[x]$. 
For a quick example of such equations see  [Perron 1954]  \cite[p. 88]{P}
and    [Brock, Elkies \& Jordan 2019] \cite{BrElJo1} for a modern treatment. 
The case  $a=0$ and $b=c=1$ was first  studied by  [Euler 1765] \cite{Eul1},  hence the name.
By the   Euler variety    $\mathscr{V}_E$  we understand  the projective closure
of  an irreducible affine variety   defined by the Euler
equations. 
The   arithmetic complexity  $c(\mathcal{A}_{RM})$ 
 is defined as  the  Krull dimension of  variety  $\mathscr{V}_E$.  
 Roughly speaking, such a  number  counts  algebraically  independent 
entries   $g_i(x)$ and  $k_j(x)$ of the continued fraction  (\ref{eq1.1.1}).   In particular,  $1\le c(\mathcal{A}_{RM})\le n$.
Our main result can be formulated as follows.

\medskip
%************************************************************************************
\begin{theorem}\label{thm1.2}
\quad  $rk~\mathcal{E}(K)=c(\mathcal{A}_{RM})-1$,  where  $\mathcal{A}_{RM}=F(\mathcal{E}(K))$.  
\end{theorem}
%************************************************************************************

\medskip\noindent
The article is organized as follows.  We introduce notation and review the known facts in Section 2.
Theorem \ref{thm1.2} is proved in Section 3.  An  illustration of  theorem \ref{thm1.2} is given 
in Section 4.

%**************************************************************************
\section{Preliminaries}
%***************************************************************************
In this section we briefly  review   $C^*$-algebras  and introduce the Mordell
$AF$-algebras.  For a general account  of   $C^*$-algebras we refer the reader to 
[Murphy   1990]  \cite{M}.   $AF$-algebras are covered in \cite[Section 3.5]{N}.

%**************************************************************************
\subsection{$AF$-algebras}
%***************************************************************************
A {\it $C^*$-algebra} is an algebra $A$ over $\mathbf{C}$ with a norm
$a\mapsto ||a||$ and an involution $a\mapsto a^*$ such that
it is complete with respect to the norm and $||ab||\le ||a||~ ||b||$
and $||a^*a||=||a^2||$ for all $a,b\in A$.
Any commutative $C^*$-algebra is  isomorphic
to the algebra $C_0(X)$ of continuous complex-valued
functions on some locally compact Hausdorff space $X$; 
otherwise, $A$ represents a non-commutative  topological
space.

An {\it $AF$-algebra}  (Approximately Finite $C^*$-algebra) is defined to
be the  norm closure of a dimension-increasing  sequence of   finite dimensional
$C^*$-algebras $M_n$,  where  $M_n$ is the $C^*$-algebra of $n\times n$ matrices
with entries in $\mathbf{C}$. Here the index $n=(n_1,\dots,n_k)$ represents
the  semi-simple matrix algebra $M_n=M_{n_1}\oplus\dots\oplus M_{n_k}$.
The ascending sequence mentioned above  can be written as 
%***********************************************************
\begin{equation}\label{eq3}
M_1\buildrel\rm\varphi_1\over\longrightarrow M_2
   \buildrel\rm\varphi_2\over\longrightarrow\dots,
\end{equation}
%****************************************************
where $M_i$ are the finite dimensional $C^*$-algebras and
$\varphi_i$ the homomorphisms between such algebras.  
The homomorphisms $\varphi_i$ can be arranged into  a graph as follows. 
Let  $M_i=M_{i_1}\oplus\dots\oplus M_{i_k}$ and 
$M_{i'}=M_{i_1'}\oplus\dots\oplus M_{i_k'}$ be 
the semi-simple $C^*$-algebras and $\varphi_i: M_i\to M_{i'}$ the  homomorphism. 
(To keep it simple, one can assume that $i'=i+1$.) 
One has  two sets of vertices $V_{i_1},\dots, V_{i_k}$ and $V_{i_1'},\dots, V_{i_k'}$
joined by  $b_{rs}$ edges  whenever the summand $M_{i_r'}$ contains $b_{rs}$
copies of the summand $M_{i_s}$ under the embedding $\varphi_i$. 
As $i$ varies, one obtains an infinite graph called the  {\it Bratteli diagram} of the
$AF$-algebra.  The matrix $B=(b_{rs})$ is known as  a {\it partial multiplicity} matrix;
an infinite sequence of $B_i$ defines a unique $AF$-algebra.

For a unital $C^*$-algebra $A$, let $V(A)$
be the union (over $n$) of projections in the $n\times n$
matrix $C^*$-algebra with entries in $A$;
projections $p,q\in V(A)$ are {\it equivalent} if there exists a partial
isometry $u$ such that $p=u^*u$ and $q=uu^*$. The equivalence
class of projection $p$ is denoted by $[p]$;
the equivalence classes of orthogonal projections can be made to
a semigroup by putting $[p]+[q]=[p+q]$. The Grothendieck
completion of this semigroup to an abelian group is called
the  $K_0$-group of the algebra $A$.
The functor $A\to K_0(A)$ maps the category of unital
$C^*$-algebras into the category of abelian groups, so that
projections in the algebra $A$ correspond to a positive
cone  $K_0^+\subset K_0(A)$ and the unit element $1\in A$
corresponds to an order unit $u\in K_0(A)$.
The ordered abelian group $(K_0,K_0^+,u)$ with an order
unit  is called a {\it dimension group};  an order-isomorphism
class of the latter is denoted by $(G,G^+)$.  The  dimension group
$(K_0(\mathbb{A}), K_0^+(\mathbb{A}), u)$ of an $AF$-algebra $\mathbb{A}$ 
is a complete  invariant of the isomorphism class  of  algebra $\mathbb{A}$.

%**************************************************************************
\subsection{Mordell $AF$-algebras}
%***************************************************************************
Denote by $S^1$ a unit circle in the complex plane;   let $G$ be a multiplicative  subgroup 
of $S^1$ given by a finite set of  generators $\{\gamma_j\}_{j=1}^s$. 
%*****************************************************************************
\begin{lemma}\label{lem2.1}
There exists a bijection  between   groups  $G\subset S^1$
and the dimension groups $(\Lambda_G, \Lambda_G^+,u)$ given by the formula:
%***************************************************************************
\begin{equation}
\left\{
\begin{array}{lll}
\Lambda_G &:=& \mathbf{Z}+\mathbf{Z}\omega_1+\dots+\mathbf{Z}\omega_s\subset\mathbf{R}, 
\quad \omega_j={1\over 2\pi}~Arg~(\gamma_j),\\
\Lambda_G^+ &:=&  \Lambda_G\cap\mathbf{R}^+,\\
u  &:=&  \hbox{\textnormal{an order unit.}}
\end{array}
\right.
\end{equation}
%****************************************************************************
The rank of $(\Lambda_G, \Lambda_G^+,u)$ is equal to $s-t+1$,
where $t$ is the total number of roots of unity among $\gamma_j$.  
\end{lemma}
%******************************************************************************
\begin{proof}
If  $\gamma_j$ is a root of unity,  then  $\omega_j={p_j\over q_j}$ is a rational 
number.  We have 
%***************************************************************************
\begin{equation}
\begin{array}{lll}
\Lambda_G &\cong& \mathbf{Z}+\mathbf{Z}{p_1\over q_1}+\dots+ \mathbf{Z}{p_t\over q_t}
+\mathbf{Z}\omega_1+\dots+\mathbf{Z}\omega_{s-t}\cong\\
&\cong& \mathbf{Z}+\mathbf{Z}\omega_1+\dots+\mathbf{Z}\omega_{s-t},
\end{array}
\end{equation}
%****************************************************************************
where $\omega_i$ are linearly indpendent irrational numbers. 
Since the set $\Lambda_G$
 is dense in $\mathbf{R}$,  the triple $(\Lambda_G, \Lambda_G^+,u)$ 
 is a dimension group [Effros 1981]   \cite[Corollary 4.7]{E}. 
 Clearly, the rank of such a group is equal to $s-t+1$. 
 A converse statement is proved similarly.  Lemma \ref{lem2.1}  follows.
\end{proof}

Let $\mathcal{E}(\mathbf{R})$   ($\mathcal{E}(\mathbf{C})$, resp.)  be the real (complex, resp.) 
points of $\mathcal{E}(K)$.  Without loss of generality, we assume that $K\subset \mathbf{R}$
 has a real embedding;  if the field $K$ is a totally imaginary CM-field, then we take its totally real
 subfield.   Thus   we have the following inclusions:
%***************************************************************************
\begin{equation}\label{eq2.4}
\mathcal{E}(K)\subset  \mathcal{E}(\mathbf{R})\subset  \mathcal{E}(\mathbf{C}). 
\end{equation}
%****************************************************************************
In view of the Mordell Theorem, we shall write $\mathcal{E}(K)\cong \mathbf{Z}^r\oplus \mathcal{E}_{tors}(K)$,
where $r=rk~\mathcal{E}(K)$ and $\mathcal{E}_{tors}(K)$ are the torsion points of $\mathcal{E}(K)$. 
It is well known that the Weierstrass $\wp$-function maps  $\mathcal{E}(\mathbf{C})$  ($\mathcal{E}(\mathbf{R})$, resp.)
to a complex torus $\mathbf{C}/\Lambda$ (the one-dimensional compact connected Lie group $S^1$, resp.)
In view of inclusion (\ref{eq2.4}) the points of  $\mathcal{E}(K)$  map  to an  abelian subgroup $G$ of $S^1$,  so that 
the $\mathcal{E}_{tors}(K)$  consists of the roots of unity,  see e.g.  [Silverman \& Tate  1992]  \cite[p.40]{ST}. 
Denote by $t$  the total number of generators of $\mathcal{E}_{tors}(K)$;  
then $\mathcal{E}(K)$ has $s=r+t$ generators.   Lemma \ref{lem2.1}  
implies a bijection between elliptic curves  $\mathcal{E}(K)$ modulo their torsion points and 
 dimension groups  $(\Lambda_G, \Lambda_G^+,u)$  of  rank $r+1$;
 the following definition is natural. 
%*******************************************************************************
\begin{definition}
By the Mordell $AF$-algebra $\mathbb{A}_{\mathcal{E}(K)}$  of an elliptic curve  $\mathcal{E}(K)$ we understand an $AF$-algebra 
given by the Elliott isomorphism:
%***************************************************************************
\begin{equation}\label{eq2.5}
K_0(\mathbb{A}_{\mathcal{E}(K)})\cong  (\Lambda_G, \Lambda_G^+,u),
\end{equation}
%****************************************************************************
where  $(\Lambda_G, \Lambda_G^+,u)$ is a dimension group of  rank $r+1$. 
\end{definition}
%******************************************************************************

%**************************************************************************
\section{Proof}
%***************************************************************************
Theorem \ref{thm1.2} follows  from an observation that the Mordell 
$AF$-algebra $\mathbb{A}_{\mathcal{E}(K)}$  is a non-commutative  
coordinate ring of  the Euler variety $\mathscr{V}_E$  defined in Section 1.  We shall split the proof 
in a series of lemmas. 
%*****************************************************************************
\begin{lemma}\label{lem3.1}
 $\mathscr{V}_E$ is a flat family of abelian varieties $\mathscr{A}_E$ over the 
projective line $\mathbf{C}P^1$.
\end{lemma}
%******************************************************************************
\begin{proof}
To prove lemma \ref{lem3.1} we shall adopt and generalize an argument of 
[Perron 1954]  \cite[Proposition 3.17]{P}. 
Let us show that there exists a flat morphism 
$p: \mathscr{V}_E\to \mathbf{C}P^1$ of 
 the Euler variety  $\mathscr{V}_E$ into  $\mathbf{C}P^1$ with the fiber $p^{-1}(x)$ 
being an  abelian variety $\mathscr{A}_E$.

Indeed, recall that the algebras $\mathcal{A}_{\theta}$ and $\mathcal{A}_{\theta'}$ are 
said to be {\it Morita equivalent}  if  $\mathcal{A}_{\theta}\otimes\mathcal{K}\cong 
\mathcal{A}_{\theta'}\otimes\mathcal{K}$, where $\mathcal{K}$ is the $C^*$-algebra of 
compact operators.  $\mathcal{A}_{\theta}$ and $\mathcal{A}_{\theta'}$ are 
Morita equivalent  if and only if $\theta'={\alpha\theta+\beta\over \gamma\theta+\delta}$, where $\alpha,\beta,\gamma,
\delta\in\mathbf{Z}$ satisfy the equality $\alpha\delta-\beta\gamma=\pm 1$
 [Rieffel 1990]  \cite{Rie1}. 
 In other words, the continued fractions of $\theta$ and $\theta'$ must coincide everywhere but a finite 
 number of entries  [Perron 1954]  \cite[Section 13]{P}.
On the other hand, ${\mathcal A}_{\theta'}\cong {\mathcal A}_{\theta}$  are isomorphic algebras if and only 
if $\theta'=\theta$ or $\theta'=1-\theta$ [Rieffel 1990]  \cite{Rie1}. 
 %********************************************************************************
\begin{remark}\label{rmk3.2}
Functor $F: \mathcal{E}(K)\to \mathcal{A}_{RM}$ maps elliptic
curves which are isomorphic over the field $K$ (over the algebraic closure $\bar K$, resp.)
to  isomorphic (Morita equivalent, resp.) non-commutative tori.
In particular,  twists of  $\mathcal{E}(K)$ correspond to the  isomorphism
classes of the  algebra $\mathcal{A}_{RM}$.
This fact follows from an  isomorphism $A(K)\otimes M_n\cong A(K^{(n)})$,
where $A(K)$ is an algebra over the field $K$, $M_n$ is the matrix algebra of rank $n$ and
$K^{(n)}$ is an extension of degree $n$ of  $K$. 
\end{remark}
%*******************************************************************************
Let $\theta_d={a+b\sqrt{d}\over c}=[g_1,\dots, g_m,  ~\overline{k_1,\dots, k_{n-m}}]$,  where 
$(k_1,\dots, k_{n-m})$ is  the minimal period  of the  continued fraction of $\theta_d$. 
To find the Euler equations for $\theta_d$, let $D=b^2d$ and  write $\theta_D$ 
in the form: 
%***************************************************************************
\begin{equation}\label{eq3.1}
\theta_D={A_n\theta_D +A_{n-1}\over B_n\theta_D+B_{n-1}},
\end{equation}
%****************************************************************************
where ${A_i\over B_i}$ is the $i$-th partial fraction of $\theta_D$
and $A_{n}B_{n-1}-A_{n-1}B_{n}=\pm 1$  [Perron 1954]  \cite[Section 19]{P};  
in particular,
%***************************************************************************
\begin{equation}\label{eq3.2}
\left\{
\begin{array}{lll}
A_n &=& k_n A_{n-1}+A_{n-2}\\
B_n &=&  k_n B_{n-1}+B_{n-2}.
\end{array}
\right.
\end{equation}
%****************************************************************************
From equation (\ref{eq3.1})  one finds that 
%***************************************************************************
\begin{equation}\label{eq3.3}
\theta_D={A_n-B_{n-1}+\sqrt{(A_n-B_{n-1})^2+4A_{n-1}B_n}\over 
2B_n}.
\end{equation}
%****************************************************************************
Comparing equation (\ref{eq3.3})  with  conditions (\ref{eq1.1}), we conclude that the Euler 
equations for $\theta_D$ have the form:  
%***************************************************************************
\begin{equation}\label{eq3.4}
\left\{
\begin{array}{ccc}
A_n-B_{n-1} &=& c_1\\
2B_n &=&  c_2\\
D &=& c_1^2+2c_2 A_{n-1},
\end{array}
\right.
\end{equation}
%****************************************************************************
where $c_1, c_2\in \mathbf{Z}$ are constants. 
Using formulas (\ref{eq3.2}) one can rewrite $c_1$ and $c_2$ in the form:
%***************************************************************************
\begin{equation}\label{eq3.5}
\left\{
\begin{array}{lll}
c_1 &=& k_n A_{n-1}+A_{n-2}-B_{n-1}\\
c_2 &=& 2 k_n B_{n-1}+2B_{n-2}.
\end{array}
\right.
\end{equation}
%****************************************************************************
We substitute $c_1$ and $c_2$ given by (\ref{eq3.5}) into the last equation
of system (\ref{eq3.4}) and get an equation:
%***************************************************************************
\begin{equation}\label{eq3.6}
D-k_n^2A_{n-1}^2-2(A_{n-1}A_{n-2}-A_{n-1}B_{n-1}+2B_{n-1})k_n=
(A_{n-2}+B_{n-1})^2 \pm 4.
\end{equation}
%****************************************************************************
But equation (\ref{eq3.6}) is a linear diophantine equation in variables $D-k_n^2A_{n-1}$
and $k_n$;  since (\ref{eq3.6}) has one solution in integer numbers, it has infinitely 
many such solutions. In other words, the projective closure $\mathscr{V}_E$ 
of an affine variety defined 
by the Euler equations (\ref{eq3.4})  is a fiber bundle over the projective line   
$\mathbf{C}P^1$.

Let us show that the fiber $p^{-1}(x)$ of a flat morphism 
$p: \mathscr{V}_E\to \mathbf{C}P^1$ is an abelian variety  $\mathscr{A}_E$. 
Indeed, the fraction $[g_1,\dots, g_m,  ~\overline{k_1,\dots, k_{n-m}}]$
corresponding to the quadratic irrationality
$\theta_d$ can be written in a normal form for which the  Euler equations 
$A_n-B_{n-1}=c_1$ and $2B_n=c_2$ in system (\ref{eq3.4}) become a trivial  
identity.  (For instance, if $c_1=0$,  then the normal fraction  has the form 
$[\gamma_1,  \overline{\kappa_1,\kappa_2,\dots,\kappa_2,\kappa_1, 2\gamma_1}]$, see e.g. 
[Perron 1954]  \cite[Section 24]{P}.)   
Denote by  $[\gamma_1,\dots, \gamma_m,  \overline{\kappa_1,\dots,\kappa_{n-m}}]$ a normal form of 
the fraction  $[g_1,\dots, g_m,  ~\overline{k_1,\dots, k_{n-m}}]$. Then
the Euler equations (\ref{eq3.4}) reduce to a single equation 
$D=c_1^2+c_2 A_{n-1}$, which is equivalent to the linear equation (\ref{eq3.6}). 
Thus the quotient of the Euler variety $\mathscr{V}_E$ by  $\mathbf{C}P^1$
corresponds to the projective closure $\mathscr{A}_E$ of an affine variety defined by $\gamma_i$ and $\kappa_j$.
But $\mathscr{A}_E$ has an obvious translation symmetry, i.e. 
$\{\gamma_i'=\gamma_i+c_i, ~\kappa_j'=\kappa_j+c_j ~|~c_i, c_j\in\mathbf{Z}\}$ define an isomorphic 
projective variety. In other words, the fiber $p^{-1}(x)\cong \mathscr{A}_E$
is  an abelian group,  i.e.  $\mathscr{A}_E$ is an abelian variety.
Lemma \ref{lem3.1} follows. 
\end{proof}

%*******************************************************************************
\begin{remark}\label{rmk3.3.1}
The family  of non-commutative tori $\mathcal{A}_{\theta_x}$ given by 
 (\ref{eq1.1})  is an analog of a rational elliptic surface or, equivalently, an elliptic curve over a one-dimensional function 
field, see e.g. [Silverman 1994]  \cite[Chapter III]{S}.  
It transpires, that $\mathscr{A}_E\cong \mathbf{C}^r/\Lambda$,  where 
$r=rk ~\mathcal{E}(K)$ and 
$\Lambda$ is the Mordell-Weil lattice of $\mathcal{E}(K)$ [Sch\"utt  \& Shioda 2019]  \cite[Chapter 1]{SS}. 
From this standpoint, the Mordell-Weil lattice corresponding to the elliptic curve $\mathcal{E}(K)$
is nothing but the universal cover of the Euler variety $\mathscr{V}_E$.
 We will not use this fact anywhere  below and will give an independent proof of
 theorem \ref{thm1.2}.  
 \end{remark}
%********************************************************************************
%*******************************************************************************
\begin{remark}\label{rmk3.3.2}
It was proved by [Brock, Elkies \& Jordan 2019] \cite{BrElJo1},  that 
$\mathscr{V}_E$ is a fiber bundle over the Fermat-Pell conic.
The structure group of the bundle was studied in \cite{Nik5}.  
\end{remark}
%********************************************************************************

%*****************************************************************************
\begin{corollary}\label{cor3.2}
$\dim_{\mathbf{C}} ~\mathscr{V}_E=1+ \dim_{\mathbf{C}} ~\mathscr{A}_E$
\end{corollary}
%**************************************************************************
\begin{proof}
Since the morphism $p: \mathscr{V}_E\to \mathbf{C}P^1$ is flat,  we conclude 
that 
%***************************************************************************
\begin{equation}\label{eq3.7}
\dim_{\mathbf{C}}\mathscr{V}_E=\dim_{\mathbf{C}} (\mathbf{C}P^1)+\dim_{\mathbf{C}} p^{-1}(x).
\end{equation}
%****************************************************************************
But  $\dim_{\mathbf{C}} (\mathbf{C}P^1)=1$ and $p^{-1}(x)\cong \mathscr{A}_E$. 
Corollary \ref{cor3.2} follows from formula (\ref{eq3.7}). 
\end{proof}

%*****************************************************************************
\begin{lemma}\label{lem3.3}
The algebra  $\mathbb{A}_{\mathcal{E}(K)}$ is a non-commutative coordinate ring of 
the abelian variety $\mathscr{A}_E$. 
\end{lemma}
%******************************************************************************
\begin{proof}
This fact follows from the equivalence of four categories
$\mathbf{E}_0\cong \mathbf{A}_0\cong \mathbf{M}\cong\mathbf{V}$,
where $\mathbf{E}_0\subset\mathbf{E}$ is a subcategory of the category
$\mathbf{E}$ of all non-singular elliptic curves consisting of the curves 
defined over a number field $K$,  $\mathbf{A}_0\subset\mathbf{A}$ 
is a subcategory of the category $\mathbf{A}$ of all non-commutative
tori consisting of the tori with real multiplication,  $\mathbf{M}$ is a 
category of all Mordell $AF$-algebras and  $\mathbf{V}$ is a category 
of all Euler varieties. The morphisms in the  categories $\mathbf{E}$ and $\mathbf{V}$
are isomorphisms between  the projective varieties and the morphisms in the
categories $\mathbf{A}$ and $\mathbf{M}$ are Morita equivalences 
(stable isomorphisms) between  the corresponding $C^*$-algebras
[Murphy   1990]  \cite{M}.  Let us pass to a detailed argument.

The equivalence $\mathbf{E}_0\cong \mathbf{A}_0$ follows from  \cite[Section 5.2]{N}. It remains to show that $\mathbf{E}_0\cong\mathbf{M}$
and $\mathbf{A}_0\cong\mathbf{V}$. 

\medskip
(i) Let us prove that $\mathbf{E}_0\cong\mathbf{M}$. 
Let $\mathcal{E}(K)\in\mathbf{E}_0$ be an elliptic curve and 
$\mathbb{A}_{\mathcal{E}(K)}\in\mathbf{M}$ be the corresponding 
Mordell $AF$-algebra, see Section 2.2.  Then 
$K_0(\mathbb{A}_{\mathcal{E}(K)})\cong  \mathcal{E}(K)$ 
by the Elliott isomorphism (\ref{eq2.5}).   
Note that the order unit $u$ in (\ref{eq2.5}) depends on the choice of 
generators of the abelian group  $\mathcal{E}(K)$ and, therefore, an isomorphism
class of  $\mathcal{E}(K)$ corresponds to the Morita equivalence class of the 
$AF$-algebra $\mathbb{A}_{\mathcal{E}(K)}$.  Thus formula (\ref{eq2.5}) 
defines an equivalence between the categories $\mathbf{E}_0$ and $\mathbf{M}$. 

\smallskip
(ii) Let us prove that $\mathbf{A}_0\cong\mathbf{V}$. 
Let $\mathcal{A}_{RM}\in\mathbf{A}_0$ be a non-commutative torus
with real multiplication and let $(k_1,\dots,k_n)$ be the minimal period 
of the continued fraction corresponding to the quadratic irrational number 
$\theta_d$. The period is  a Morita invariant of the algebra $\mathcal{A}_{RM}$ and a
cyclic permutation of $k_i$ gives an  algebra $ \mathcal{A}_{RM}'$
which is Morita equivalent  to $\mathcal{A}_{RM}$. On the other hand,
it is not hard to see that the period $(k_1,\dots,k_n)$ can be uniquely recovered from the 
coefficients of the Euler equations (\ref{eq3.4}) and a cyclic permutation of $k_i$ 
in (\ref{eq3.4}) defines an Euler variety $\mathscr{V}_E'$ which is isomorphic to 
 $\mathscr{V}_E$. In other words, the categories $\mathbf{A}_0$ and $\mathbf{V}$
are isomorphic. 

\medskip
Comparing (i) and (ii) with the equivalence $\mathbf{E}_0\cong\mathbf{A}_0$,
one concludes that $\mathbf{V}\cong\mathbf{M}$ are equivalent categories,
where a functor $F: \mathbf{V}\to\mathbf{M}$ acts by the formula  
$\mathscr{V}_E\mapsto \mathbb{A}_{\mathcal{E}(K)}$ 
given by the closure of arrows of the commutative diagram in Figure 1. 
It remains to recall that  $\mathscr{V}_E$ is a fiber bundle over the 
$\mathbf{C}P^1$; therefore all geometric data of  $\mathscr{V}_E$ is 
recorded by the fiber $\mathscr{A}_E$ alone. We conclude that the 
algebra $\mathbb{A}_{\mathcal{E}(K)}$ is a twisted homogeneous coordinate ring of the 
abelian variety  $\mathscr{A}_E$, i.e. an isomorphism class of 
 $\mathscr{A}_E$  corresponds to the Morita equivalence class 
 of the algebra  $\mathbb{A}_{\mathcal{E}(K)}$. Lemma \ref{lem3.3}
 follows. 
\end{proof}

%*******************************************************************
\begin{figure}
%[here]
%*******************************************************************
\begin{picture}(300,110)(-80,-5)
\put(20,70){\vector(0,-1){35}}
\put(105,70){\vector(0,-1){35}}
\put(40,30){\vector(1,1){50}}
\put(40,78){\vector(1,-1){50}}
\put(10,20){$\mathcal{A}_{RM}$}
\put(100,18){$\mathbb{A}_{\mathcal{E}(K)}$}
\put(10,80){$\mathcal{E}(K)$}
\put(100,80){$\mathscr{V}_E\cong (\mathscr{A}_E, \mathbf{C}P^1, p)$}
\put(115,50){$F$}
\end{picture}
%******************************************************************
\caption{Functor  $F: \mathbf{V}\to\mathbf{M}$.}
\end{figure}
%*******************************************************************

%******************************************************************
\begin{remark}\label{rmk3.4}
The diagram in Figure 1 implies an equivalence of the categories: 
(i) $\mathbf{E}_0\cong\mathbf{V}$ and (ii) $\mathbf{A}_0\cong\mathbf{M}$. 
The bijection (i) is realized by the formula $\mathcal{E}(K)\mapsto
\mathscr{A}_E\cong Jac~(X(N))$,  where $X(N)$ is an Eichler-Shimura-Taniyama
modular curve of level $N$ over $\mathcal{E}(K)$  and $Jac~(X(N))$ is the 
Jacobian of $X(N)$.  The bijection (ii) is given by the formula 
$\mathcal{A}_{RM}\mapsto \mathbb{A}_{\mathcal{E}(K)}$,  so that 
the number field $K_0(\mathbb{A}_{\mathcal{E}(K)})\otimes\mathbf{Q}$ 
is the maximal abelian extension of a real quadratic number field 
$K_0(\mathcal{A}_{RM})\otimes\mathbf{Q}$. 
\end{remark}
%******************************************************************
%*****************************************************************************
\begin{lemma}\label{lem3.5}
$rk~\mathcal{E}(K)=\dim_{\mathbf{C}} ~\mathscr{A}_E$.
\end{lemma}
%******************************************************************************
\begin{proof}
Let $\mathscr{A}_E$ be an abelian variety  over the field $\mathbf{Q}$ 
and let $n=\dim_{\mathbf{C}} ~\mathscr{A}_E$.
In this case it is known that the coordinate ring of $\mathscr{A}_E$
is isomorphic to the algebra $\mathcal{A}_{RM}^{2n}$, where 
$\mathcal{A}_{RM}^{2n}$ is a $2n$-dimensional non-commutative 
torus with real multiplication \cite[Remark 6.6.2]{N}.
(Notice that the proof can be easily extended from the case of
  complex multiplication to the case of $\mathscr{A}_E$
  over $\mathbf{Q}$.)  Moreover, 
%*********************************************************************************
\begin{equation}\label{eq3.8}
K_0(\mathcal{A}_{RM}^{2n})\cong \mathbf{Z}+\mathbf{Z}\theta_1+\dots+\mathbf{Z}\theta_n, 
\end{equation}
%*********************************************************************************  
 see  \cite[Remark 6.6.2]{N}. 
But we know from lemma \ref{lem3.3}  that a coordinate ring of the abelian variety 
$\mathscr{A}_E$ is isomorphic to the algebra $\mathbb{A}_{\mathcal{E}(K)}$, 
so that 
%*********************************************************************************
\begin{equation}\label{eq3.9}
rk ~K_0(\mathbb{A}_{\mathcal{E}(K)})=1+rk ~\mathcal{E}(K), 
\end{equation}
%*********************************************************************************  
see Section 2.2.  Comparing formulas (\ref{eq3.8}) and (\ref{eq3.9}),
we conclude that $rk~\mathcal{E}(K)=n=\dim_{\mathbf{C}} ~\mathscr{A}_E$.
Lemma \ref{lem3.5} follows. 
\end{proof}

\bigskip
Theorem \ref{thm1.2} follows from corollary \ref{cor3.2}, lemma \ref{lem3.5}
and a definition of the arithmetic complexity $c(\mathcal{A}_{RM})=\dim_{\mathbf{C}} ~\mathscr{V}_E$.

%**************************************************************************
\section{Examples}
%***************************************************************************
We shall  illustrate   theorem \ref{thm1.2} by  an example 
of  $\mathbf{Q}$-curves introduced in [Gross  1980] \cite{G}. 
Denote by   $\mathcal{E}_{CM}^{(-d,f)}$
an  elliptic curve with  complex multiplication by an order $R_f$ of conductor $f$
in the imaginary quadratic field  $\mathbf{Q}(\sqrt{-d})$,  see e.g. [Silverman 1994] 
\cite[pp. 95-96]{S}.   It is known that  the non-commutative torus corresponding 
to $\mathcal{E}_{CM}^{(-d,f)}$
will have  real multiplication by  the order $\mathfrak{R}_{\mathfrak{f}}$ of 
conductor $\mathfrak{f}$  in the quadratic field $\mathbf{Q}(\sqrt{d})$;  such a torus
we denote by  $\mathcal{A}_{RM}^{(d, \mathfrak{f})}$.    The conductor $\mathfrak{f}$ 
is defined by the equation $|Cl~(\mathfrak{R}_{\mathfrak{f}})|=|Cl~(R_f)|$,  where $Cl$ is the
class  group of the respective orders  \cite[Section 6.1.1]{N}.   
Let $(\mathcal{E}_{CM}^{(-d,f)})^{\sigma},  ~\sigma\in Gal~(k| \mathbf{Q})$ be 
the  Galois conjugate of the curve  $\mathcal{E}_{CM}^{(-d,f)}$; 
by a {\it $\mathbf{Q}$-curve} one understands   $\mathcal{E}_{CM}^{(-d,f)}$,
such that  there exists  an  isogeny between $(\mathcal{E}_{CM}^{(-d,f)})^{\sigma}$ and   $\mathcal{E}_{CM}^{(-d,f)}$
for  each    $\sigma\in Gal~(k|\mathbf{Q})$. 
Let $\mathfrak{P}_{3 \mod 4}$ be the set of all primes  
$p=3 \mod 4$;    it is known that  $\mathcal{E}_{CM}^{(-p,1)}$ is a $\mathbf{Q}$-curve 
whenever $p\in  \mathfrak{P}_{3 ~mod ~4}$  [Gross 1980]   \cite[p. 33]{G}. 
The rank of  $\mathcal{E}_{CM}^{(-p,1)}$  is always divisible by $2h_K$,  where $h_K$ is 
the  class number of the field $K:=\mathbf{Q}(\sqrt{-p})$  [Gross 1980]   \cite[p. 49]{G};  
  by a {\it $\mathbf{Q}$-rank}  of  $\mathcal{E}_{CM}^{(-p,1)}$
 we understand the integer 
 $rk_{\mathbf{Q}}(\mathcal{E}_{CM}^{(-p,1)}):={1\over 2h_K}~rk~(\mathcal{E}_{CM}^{(-p,1)})$.   
Table 1  below shows the $\mathbf{Q}$-ranks of the $\mathbf{Q}$-curves
and the arithmetic complexity of the corresponding non-commutative tori 
for the primes $1<p<100$. The respective calculations of the rank can be found in 
[Gross 1980]   \cite{G} and of  the arithmetic complexity
in \cite[Section 6.2.4.1]{N}.   It is easy to see,  that in all the examples of Table 1 the rank of elliptic curves and
the corresponding arithmetic complexity satisfy the equality: 
%*********************************************************************************
\begin{equation}\label{eq4.0}
 rk_{\mathbf{Q}}\left(\mathcal{E}_{CM}^{(-p,1)}\right)=c\left(\mathcal{A}_{RM}^{(p,1)}\right)-1,
\end{equation}
%*********************************************************************************  
where $\mathcal{A}_{RM}^{(p,1)}=F(\mathcal{E}_{CM}^{(-p,1)})$. 

%************************************************************
\begin{table}[h]
\begin{tabular}{c|c|c|c}
\hline
&&&\\
$p\equiv 3 \mod 4$ & $rk_{\mathbf{Q}}(\mathcal{E}_{CM}^{(-p,1)})$ & $\sqrt{p}$ & 
$c(\mathcal{A}_{RM}^{(p,1)})$\\
&&&\\
\hline
$3$ & $1$ & $[1,\overline{1,2}]$ & $2$\\
\hline
$7$ & $0$ & $[2,\overline{1,1,1,4}]$ & $1$\\
\hline
$11$ & $1$ & $[3,\overline{3,6}]$ & $2$\\
\hline
$19$ & $1$ & $[4,\overline{2,1,3,1,2,8}]$ & $2$\\
\hline
$23$ & $0$ & $[4,\overline{1,3,1,8}]$ & $1$\\
\hline
$31$ & $0$ & $[5,\overline{1,1,3,5,3,1,1,10}]$ & $1$\\
\hline
$43$ & $1$ & $[6,\overline{1,1,3,1,5,1,3,1,1,12}]$ & $2$\\
\hline
$47$ & $0$ & $[6,\overline{1,5,1,12}]$ & $1$\\
\hline
$59$ & $1$ & $[7,\overline{1,2,7,2,1,14}]$ & $2$\\
\hline
$67$ & $1$ & $[8,\overline{5,2,1,1,7,1,1,2,5,16}]$ & $2$\\
\hline
$71$ & $0$ & $[8,\overline{2,2,1,7,1,2,2,16}]$ & $1$\\
\hline
$79$ & $0$ & $[8,\overline{1,7,1,16}]$ & $1$\\
\hline
$83$ & $1$ & $[9,\overline{9,18}]$ & $2$\\
\hline
\end{tabular}

\bigskip
\centerline{Table 1.  The $\mathbf{Q}$-curves $\mathcal{E}_{CM}^{(-p,1)}$  with  $1<p<100$.}
\end{table}
%************************************************************

%%%%%%%%%%%%%%
\section*{Conflict of interest}
%%%%%%%%%%%%%%%
On behalf of all authors, the corresponding author states that there is no conflict of interest.
  
  %%%%%%%%%%%%%%
\section*{Data availability}
%%%%%%%%%%%%%%%
  
  Data sharing not applicable to this article as no datasets were generated or analyzed during the current study.

\bibliographystyle{amsplain}

%**********************************************************

\end{document}